\documentclass[11pt]{amsart}

\usepackage{graphicx}
\usepackage{amscd}
\usepackage{amssymb, amsmath, amsthm, amsfonts}
\usepackage{graphics}
\usepackage{epsfig,psfrag}
\usepackage[english]{babel}
\usepackage{latexsym}
\usepackage{mathrsfs, subfig, color}

\usepackage[margin=0.9in]{geometry}

\newtheorem{theorem}{Theorem}[section]

\newtheorem{lemma}[theorem]{Lemma}

\newtheorem{proposition}[theorem]{Proposition}
\newtheorem{conjecture}[theorem]{Conjecture}
\newtheorem{question}[theorem]{Question}

\theoremstyle{definition}

\theoremstyle{remark}
\newtheorem{remark}[theorem]{Remark}

\usepackage{mathtools}

\DeclarePairedDelimiter\floor{\lfloor}{\rfloor}

\begin{document}

\title[Counterexamples to the quadrisecant approximation conjecture]
{Counterexamples to the quadrisecant approximation conjecture}

\author{Sheng Bai}
\address{LMAM, School of Mathematical Sciences, Peking University, Beijing 100871, China}
\email{barries@163.com}

\author{Chao Wang}
\address{School of Mathematical Sciences, University of
Science and Technology of China, Hefei 230026, China}
\email{chao\_{}wang\_{}1987@126.com}

\author{Jiajun Wang}
\address{LMAM, School of Mathematical Sciences, Peking University, Beijing 100871, China}
\email{wjiajun@pku.edu.cn}

\subjclass[2010]{57M25}

\keywords{quadrisecant, trivial knot, trefoil knot, polygonal knot, edge number}

\thanks{The second named author was partially supported by Project Funded by China Postdoctoral Science Foundation (Grant No. 2015M571929). The third named author is partially supported by NSFC 11425102.}

\begin{abstract}
A quadrisecant of a knot is a straight line intersecting the knot at four points. If a knot has finitely many quadrisecants, one can replace each subarc between two adjacent secant points by the line segment between them to get the quadrisecant approximation of the original knot. It was conjectured that the quadrisecant approximation is always a knot with the same knot type as the original knot. We show that every knot type contains two knots, the quadrisecant approximation of one knot has self intersections while the quadrisecant approximation of the other knot is a knot with different knot type.
\end{abstract}

\date{}
\maketitle

\section{Introduction}

A quadrisecant for a knot is a line that intersects the knot in four points. In 1933, Pannwitz \cite{Pannwitz:1933ks} showed that a generic polygonal knot in any nontrivial knot type must have at least two quadrisecants. This result was extended to smooth knots by Morton and Mond \cite{Morton:1982dn} and to tame knots by Kuperberg \cite{Kuperberg:1994et}. Quadrisecants was used to give lower bounds of the ropelength of a knot in \cite{Denne:2006cz}.

The quadrisecant approximation of a knot was introduced by Jin in \cite{Jin:2005}.  For a knot $K$ with finitely many quadrisecants, let $W$ be the set of intersection points of $K$ with the quadrisecants. The \emph{quadrisecant approximation} of $K$, denoted by $\widehat{K}$, is obtained from $K$ by replacing each subarc of $K$ between two points in $W$ adjacent along $K$ with the straight line segment between them. If $K$ has no quadrisecants (which has to be an unknot), we let $\widehat{K}=K$. A knot as the union of finitely many line segments is called a polygonal knot. It was showed in \cite{Jin:2005} that almost every polygonal knot has only finitely many quadrisecants (see also \cite{Budney:2005ht, CruzCota:2015uv}). In every knot type, there is a polygonal knot such that the quadrisecant approximation has the same knot type. In fact, the following conjecture was proposed in the same paper:

\begin{conjecture}[The quadrisecant approximation conjecture]\label{conj:ori}
If $K$ has finitely many quadrisecants, then its quadrisecant approximation $\widehat{K}$ has the same knot type as $K$. Furthermore $K$ and $\widehat{K}$ have the same set of quadrisecants.
\end{conjecture}

The conjecture were verified for some knots with crossing number not bigger than five in \cite{Jin:2005} and for hexagonal trefoil knots in \cite{Jin:2011be}.

In the present paper, we work on polygonal knots. For a polygonal knot, let $e(K)$ be the number of edges in $K$. For a knot type $\mathcal{K}$, the edge number $e(\mathcal{K})$ is the smallest number of edges among all polygonal knots with type $\mathcal{K}$. We show the following:

\begin{theorem}\label{thm:main}
For any knot type $\mathcal{K}$, there exist polygonal knots $K_\ast$ and $K_\diamond$ of type $\mathcal{K}$ such that the quadrisecant approximation $\widehat{K_\ast}$ has self-intersections while $\widehat{K_\diamond}$ is a knot with knot type different from $\mathcal{K}$.

Furthermore, we can require that $e(K_\ast)\leq e(\mathcal{K})+6$ and $e(K_\diamond)\leq \frac52e(\mathcal{K})+17$. If $\mathcal{K}$ does not contain the trefoil knot as a connected summand,  we can require that $e(K_\diamond)\leq e(\mathcal{K})+14$.
\end{theorem}

Hence Conjecture \ref{conj:ori} does not hold in general. However, in view that our counterexamples contain redundant edges, and with the results in \cite{Jin:2011be}, the following weaker conjecture may still hold

\begin{conjecture}\label{conj:re}
For a polygonal knot $K$ of type $\mathcal{K}$ with $e(\mathcal{K})$ edges and finitely many quadrisecants, the quadrisecant approximation $\widehat{K}$ has knot type $\mathcal{K}$.
\end{conjecture}

Given a knot type $\mathcal{K}$, let $R_\ast(\mathcal{K})$ be the minimal number $e(K)-e(\mathcal{K})$ for polygonal knots $K$ with type $\mathcal{K}$ such that $\widehat{K}$ has self-intersections, and let $R_\diamond(\mathcal{K})$ be the minimal number $e(K)-e(\mathcal{K})$ for polygonal knots $K$ with type $\mathcal{K}$ such that $\widehat{K}$ is a knot with type different from $\mathcal{K}$. Theorem \ref{thm:main} shows that $R_\ast(\mathcal{K})\leq 6$ and $R_\diamond(\mathcal{K})\leq \frac32e(\mathcal{K})+17$ for general knots and $R_\diamond(\mathcal{K})\leq14$ for $\mathcal{K}$ which does not contain the trefoil knot as a connected summand. Hence the following question is natural:

\begin{question}
For a given knot type $\mathcal{K}$, what are $R_\ast(\mathcal{K})$ and $R_\diamond(\mathcal{K})$? In particular, are $R_\ast(\mathcal{K})$ and $R_\diamond(\mathcal{K})$ always positive?
\end{question}

The second question is equivalent to Conjecture \ref{conj:re}. On the other hand, in our examples, the knot type $\widehat{K_\diamond}$ is the connected sum of  $K_\diamond$ with the trefoil knots. Hence we have the following question:

\begin{question}
Given a knot type $\mathcal{K}$, what knot types can be given by $\widehat{K}$ for knots $K$ with type $\mathcal{K}$? Is the knot type of $\widehat{K}$ always a connected sum for which $\mathcal{K}$ is a summand? Can $\widehat{K}$ be in some sense ``simpler" than $K$? In particular, can the quadrisecant approximation $\widehat{K}$ be the unknot for a knot $K$ with nontrivial knot type?
\end{question}

We remark that the examples in Theorem \ref{thm:main} can be modified to give corresponding examples in the smooth category, for example, by smoothing the corners.

The paper is organized as follows. In Section \ref{sec:algorithm} we will explain how to find all quadrisecants of a given knot. In Section \ref{sec:unknot} we will give two trivial knots $K_6$ and $K_{14}$ with $6$ and $14$ edges, such that $\widehat{K_6}$ has self-intersections while $\widehat{K_{14}}$ is a trefoil knot. In Section \ref{sec:general} we will use a carefully defined connected sum operation to generalize the examples to show Theorem \ref{thm:main}.

\section{How to find quadrisecants}\label{sec:algorithm}

In this section, we give an algorithm to find all quadrisecants of a given polygonal knot. The results will be used in the construction of counterexamples in the following sections.

Let $K=V_1V_2\cdots V_n$ be a polygonal knot in $\mathbb{R}^3$. The case when $n=3$ is trivial, hence we always assume $n>3$. Let $V_{n+1}=V_1$, and for $1\leq i\leq n$, let $v_i=V_{i+1}-V_i$. We require that $K$ is in general position, namely $K$ satisfies following
\begin{enumerate}
\item[(a)] no four vertices of $K$ are coplanar;
\item[(b)] the vectors for any three edges of $K$ are linear independent.
\end{enumerate}

Given three vectors $u, v, w$, let $Det(u,v,w)$ be their determinant. Then the above two conditions are equivalent to the following:
\begin{enumerate}
\item[$(a')$] for $1\leq i<j<k<l\leq n$, $$Det(V_j,V_k,V_l)-Det(V_i,V_k,V_l)+Det(V_i,V_j,V_l)-Det(V_i,V_j,V_k)\neq0;$$
\item[$(b')$] for $1\leq i<j<k\leq n$, $Det(v_i,v_j,v_k)\neq0$.
\end{enumerate}

If $L$ is a quadrisecant of $K$, then there are four edges $V_iV_{i+1}$, $V_jV_{j+1}$, $V_kV_{k+1}$, $V_lV_{l+1}$, $1\leq i<j<k<l\leq n$, such that for $h\in \{i,j,k,l\}$, we have
$$L\cap V_hV_{h+1}\in V_hV_{h+1}-\{V_{h+1}\},$$
namely the intersection point is in the interior of the edge or is the initial point of the edge. Then the intersection points can be presented as
\begin{align}
V_i+pv_i,V_j+qv_j,V_k+rv_k,V_l+sv_l, 0\leq p,q,r,s<1.
\end{align}
Since the four points are collinear, there exist $x,y\in \mathbb{R}-\{0,1\}$ such that
\begin{align}
&(1-x)(V_i+pv_i)+x(V_j+qv_j)-(V_k+rv_k)=0,\\
&(1-y)(V_i+pv_i)+y(V_j+qv_j)-(V_l+sv_l)=0.
\end{align}
By $(b')$ and the Cramer's Rule, we have
\begin{align}
p=&\frac{Det(V_k-V_j,v_j,v_k)}{Det(v_i,v_j,v_k)}\frac{1}{1-x}+\frac{Det(V_j-V_i,v_j,v_k)}{Det(v_i,v_j,v_k)},\label{eq:x1}\\
q=&\frac{Det(v_i,V_i-V_k,v_k)}{Det(v_i,v_j,v_k)}\frac{x-1}{x}+\frac{Det(v_i,V_k-V_j,v_k)}{Det(v_i,v_j,v_k)},\label{eq:x2}\\
r=&\frac{Det(v_i,v_j,V_j-V_i)}{Det(v_i,v_j,v_k)}\frac{x}{1}+\frac{Det(v_i,v_j,V_i-V_k)}{Det(v_i,v_j,v_k)},\label{eq:x3}
\end{align}
\begin{align}
p=&\frac{Det(V_l-V_j,v_j,v_l)}{Det(v_i,v_j,v_l)}\frac{1}{1-y}+\frac{Det(V_j-V_i,v_j,v_l)}{Det(v_i,v_j,v_l)},\label{eq:y1}\\
q=&\frac{Det(v_i,V_i-V_l,v_l)}{Det(v_i,v_j,v_l)}\frac{y-1}{y}+\frac{Det(v_i,V_l-V_j,v_l)}{Det(v_i,v_j,v_l)},\label{eq:y2}\\
s=&\frac{Det(v_i,v_j,V_j-V_i)}{Det(v_i,v_j,v_l)}\frac{y}{1}+\frac{Det(v_i,v_j,V_i-V_l)}{Det(v_i,v_j,v_l)}.\label{eq:y3}
\end{align}
Let $f(z)=1/(1-z)$, then $f\circ f(z)=(z-1)/z$ and $f\circ f\circ f(z)=z$. Hence $(\ref{eq:x1})(\ref{eq:x2})(\ref{eq:x3})$ and $(\ref{eq:y1})(\ref{eq:y2})(\ref{eq:y3})$ are symmetric. Since $1\leq i<j<k<l\leq n$, $V_i,V_{i+1},V_k,V_{k+1}$ are four distinct points and $V_j,V_{j+1},V_l,V_{l+1}$ are four distinct points. By $(a')$, we have
\begin{align}
Det(v_i,V_i-V_k,v_k)\neq0, \quad Det(V_l-V_j,v_j,v_l)\neq0.
\end{align}
Hence by $(\ref{eq:x1})(\ref{eq:y1})$ and $(\ref{eq:x2})(\ref{eq:y2})$, $x$ and $y$ are dependent on each other. And we can solve $x,y$, then get $p,q,r,s$. Actually, $x$ satisfies a quadratic equation
\begin{align}
Ax^2+Bx+C=0.\label{eq:ABC}
\end{align}
Here $A$, $B$ and $C$ are some polynomial functions of coordinates of the vertices $V_i$, $V_{i+1}$, $V_j$, $V_{j+1}$, $V_k$, $V_{k+1}$, $V_l$ and $V_{l+1}$. And $A$ or $B$ or $C$ may be zero.

In $(\ref{eq:x1})(\ref{eq:x2})(\ref{eq:x3})$, let $x$ vary in $\mathbb{R}-\{0,1\}$, then we get a ruled surface $S$ which is the union of lines passing through the corresponding $V_i+pv_i$, $V_j+qv_j$ and $V_k+rv_k$. If
\begin{align*}
Det(V_k-V_j,v_j,v_k)=0, \quad Det(v_i,v_j,V_j-V_i)=0,
\end{align*}
then by $(a')$, we have $V_{j+1}=V_k$ and $V_{i+1}=V_j$. Then $p=1$, $r=0$, and $L$ can not be a quadrisecant. If only one of the determinants is zero, then $S$ is a part of a plane, and $V_i+pv_i$ or $V_k+rv_k$ is fixed when $x$ varies. If both determinants are nonzero, then one can check that $S$ is a quadric.

Hence if the equation $(\ref{eq:ABC})$ has infinitely many solutions, then $A=B=C=0$, and the four edges $V_iV_{i+1}$, $V_jV_{j+1}$, $V_kV_{k+1}$, $V_lV_{l+1}$ must be linear dependent or lie on a quadric generated by $V_iV_{i+1}$, $V_jV_{j+1}$ and $V_kV_{k+1}$. This is also the result in \cite{Jin:2005} that, for four edges in general position, there are at most two straight lines intersecting each of them, corresponding to two solutions of $(\ref{eq:ABC})$ with $A\neq 0$.

By the above discussion, for a given polygonal knot $K$, we have an algorithm to find all its quadrisecants, that is, to find all possible $0\leq p,q,r,s<1$ for every four edges of $K$. If $K$ is in general position and no edge of $K$ lies on a quadric generated by other three edges, then $K$ have finitely many quadrisecants. We can then get its quadrisecant approximation. In the following section, we will use this algorithm to determine the quadrisecant approximation of a given polygonal knot. Most computations in the present paper are performed by Mathematica 6.0.

\section{Quadrisecant approximation of the unknot}\label{sec:unknot}

In this section, we will give two polygonal unknots $K_6$ and $K_{14}$ with $6$ and $14$ edges respectively such that $\widehat{K_6}$ has self-intersections, and $\widehat{K_{14}}$ is a trefoil knot.

\subsection{Construction of $K_6$}

The $6$-edge unknot $K_6$ is constructed as follows. Let $K=V_1V_2\cdots V_6$ be the polygonal knot with the following coordinates:
\begin{align*}
&V_1=(0, 0, 0),&&V_2=(1, 0, 0),&&V_3=(2, 0, 1),\\
&V_4=(3, 0, 0),&&V_5=(4, 0, 0),&&V_6=(2, 3, 0).
\end{align*}
$K$ is obtained from the triangle $V_1V_5V_6$  in Figure \ref{fig:triangle}(a) by replacing a line segment in $V_1V_5$ by two edges $V_2V_3$ and $V_3V_4$ as in Figure \ref{fig:triangle}(b). Clear $V_1$, $V_2$, $V_4$, $V_5$ are in a straight line $\ell$. Then we can extend $V_6V_1$, $V_3V_2$, $V_3V_4$, $V_6V_5$ to get $K_6=W_1W_2\cdots W_6$ in Figure \ref{fig:triangle}(c). With suitable choices of extensions, $\ell$ will be the only quadrisecant of $K_6$, and $\widehat{K_6}$ will have self-intersections. The following is such a choice:
\begin{align*}
&W_1=(-\frac{1}{5}, -\frac{3}{10}, 0),&&W_2=(\frac{4}{5}, 0, -\frac{1}{5}),&&W_3=(2, 0, 1),\\
&W_4=(\frac{13}{4}, 0, -\frac{1}{4}),&&W_5=(\frac{17}{4}, -\frac{3}{8}, 0),&&W_6=(2, 3, 0).
\end{align*}

\begin{figure}[h]
\centerline{\scalebox{0.59}{\includegraphics{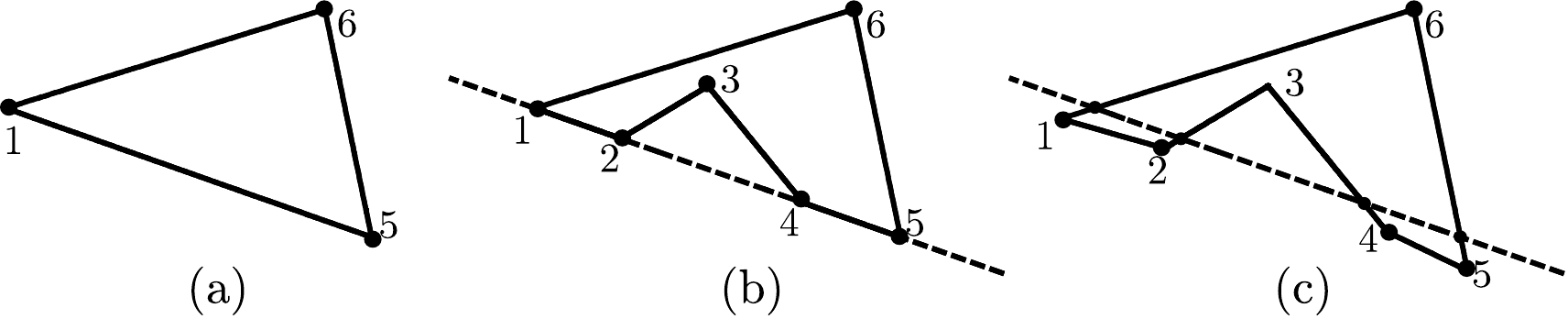}}}
\caption{The hexagonal unknot $K_6$}\label{fig:triangle}
\end{figure}

\subsection{Construction of $K_{14}$}
We first construct a primary knot $K_0$, then perturb it to get the final knot $K_{14}$. Certain lines intersecting the primary knot $K_0$ will become the quadrisecants of $K_{14}$.

The primary knot $K_0=V_1V_2\cdots V_{14}$ is the polygonal knot with the following coordinates, see Figure \ref{fig:knot0}.
\begin{align*}
&V_1=(0, 0, 0),&&V_2=(0, 0, -4),&&V_3=(8, -2, -4),\\
&V_4=(6, -3, -6),&&V_5=(0, 0, -6),&&V_6=(0, 0, -8),\\
&V_7=(10, -1, -8),&&V_8=(10, -1, 1),&&V_9=(6, 1, 0),\\
&V_{10}=(8, 0, -1),&&V_{11}=(6, -1, 0),&&V_{12}=(12, -2, -3),\\
&V_{13}=(12, 0, 0),&&V_{14}=(6, 2, 0).&&
\end{align*}

\begin{figure}[h]
\centerline{\scalebox{0.6}{\includegraphics{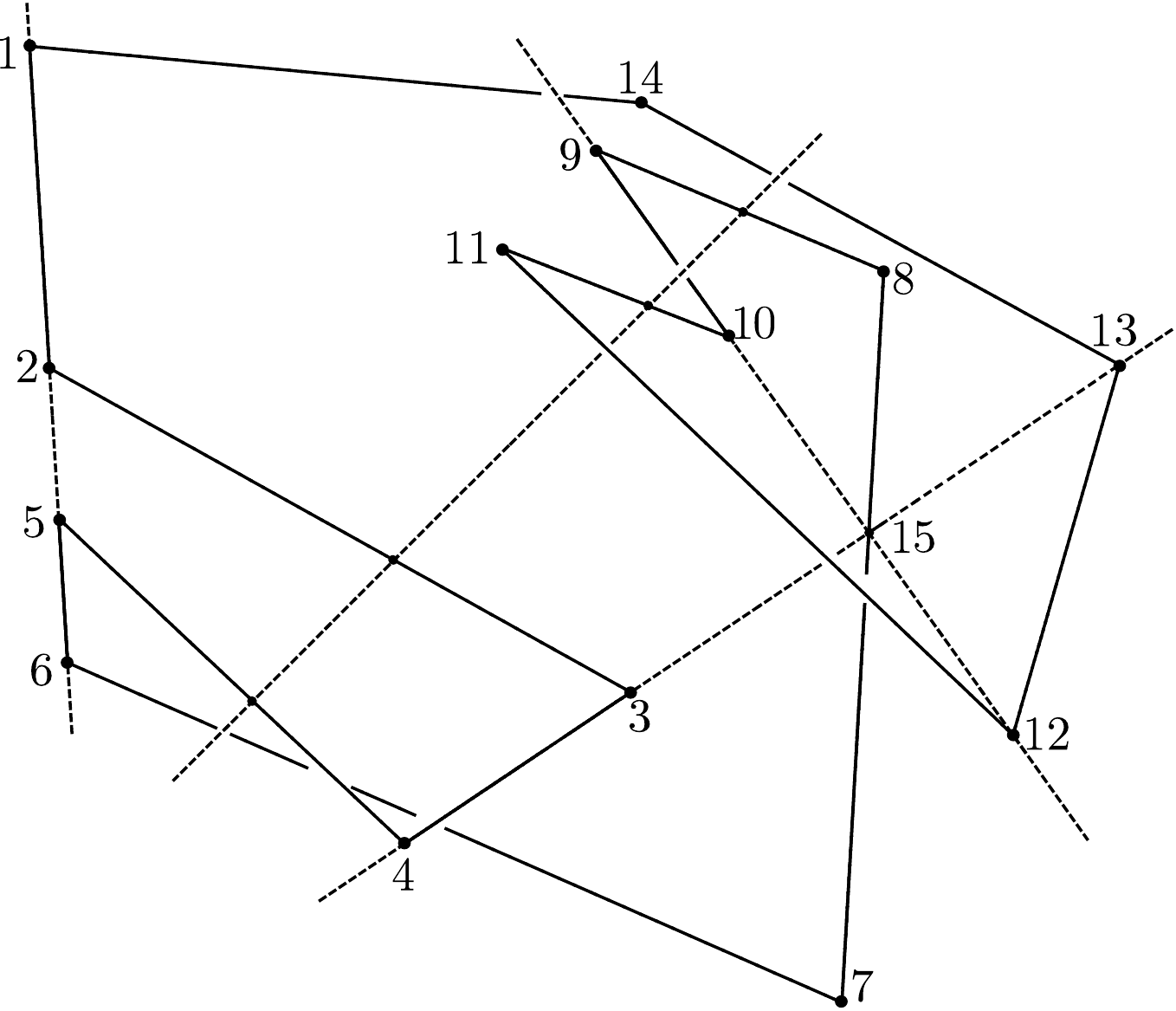}}}
\caption{The primary knot $K_0$}\label{fig:knot0}
\end{figure}

Let $V_{15}=(10,-1,-2)$. The knot $K_0$ is obtained as follows:
\begin{enumerate}
\item[(i)] choose $V_1$, $V_{13}$, $V_{14}$ and $V_9$ in the $xy$-plane, and $V_2$, $V_5$, $V_6$ in the $z$-axis;
\item[(ii)] choose $V_4$ having the same $z$-coordinate as $V_5$, and $V_3$, $V_{15}$ in $V_4V_{13}$;

\item[(iii)] choose $V_7V_8$ containing $V_{15}$ and parallel to the $z$-axis;

\item[(iv)] choose $V_{10}$, $V_{12}$ in the line passing through $V_9$ and $V_{15}$. Then choose $V_{11}$.
\end{enumerate}

The following $(a)$ and $(b)$ can be verified:

$(a)$ There are three straight lines $L_1$, $L_2$ and $L_3$ passing the four vertices in the three sets $\{V_1, V_2, V_5, V_6\}$, $\{V_3, V_4, V_{13}, V_{15}\}$ and $\{V_9, V_{10}, V_{12}, V_{15}\}$ separately;

$(b)$ There is a quadrisecant $L_4$ of $K_0$ intersecting $V_2V_3$, $V_4V_5$, $V_8V_9$ and $V_{10}V_{11}$.

In Figure \ref{fig:knot0}, $L_1$, $L_2$, $L_3$, $L_4$ are given by the dashed lines. Let $\Lambda$ be the broken line $V_1V_{14}V_{13}$. We hope that after small perturbation of the vertices, $L_1$, $L_2$, $L_3$ can become quadrisecants and no quadrisecant other than $L_1$ and $L_2$ can intersect $\Lambda$. By the discussion in Section \ref{sec:algorithm}, there will be a quadrisecant $L_4'$ intersecting $V_2V_3$, $V_4V_5$, $V_8V_9$ and $V_{10}V_{11}$. $L_4'$ can be thought as obtained from $L_4$ by a slightly movement. Then the quadrisecant approximation will be a trefoil knot.

\vspace{5pt}

The final knot $K_{14}=W_1W_2\cdots W_{14}$ is given by the following coordinates, see Figure \ref{fig:knot1}.
\begin{align*}
&W_1=(-\frac{3}{5},-\frac{1}{5},0),&&W_2=(-\frac{19}{25},\frac{11}{50},-\frac{99}{25}),&&W_3=(\frac{228}{25},-\frac{66}{25},-\frac{112}{25}),\\
&W_4=(\frac{143}{20},-\frac{121}{40},-\frac{109}{20}),&&W_5=(-\frac{13}{10},\frac{11}{20},-\frac{61}{10}),&&W_6=(-1,\frac{1}{10},-\frac{171}{20}),\\
&W_7=(10,-1,-8),&&W_8=(10,-1,1),&&W_9=(\frac{28}{5},\frac{6}{5},-\frac{1}{10}),\\
&W_{10}=(\frac{81}{10},\frac{1}{20},-\frac{21}{20}),&&W_{11}=(\frac{59}{10},-\frac{21}{20},\frac{1}{20}),&&W_{12}=(12,-\frac{11}{5},-\frac{33}{10}),\\
&W_{13}=(12,\frac{1}{5},\frac{3}{10}),&&W_{14}=(6, 2, 0).&&
\end{align*}

\begin{figure}[h]
\centerline{\scalebox{0.6}{\includegraphics{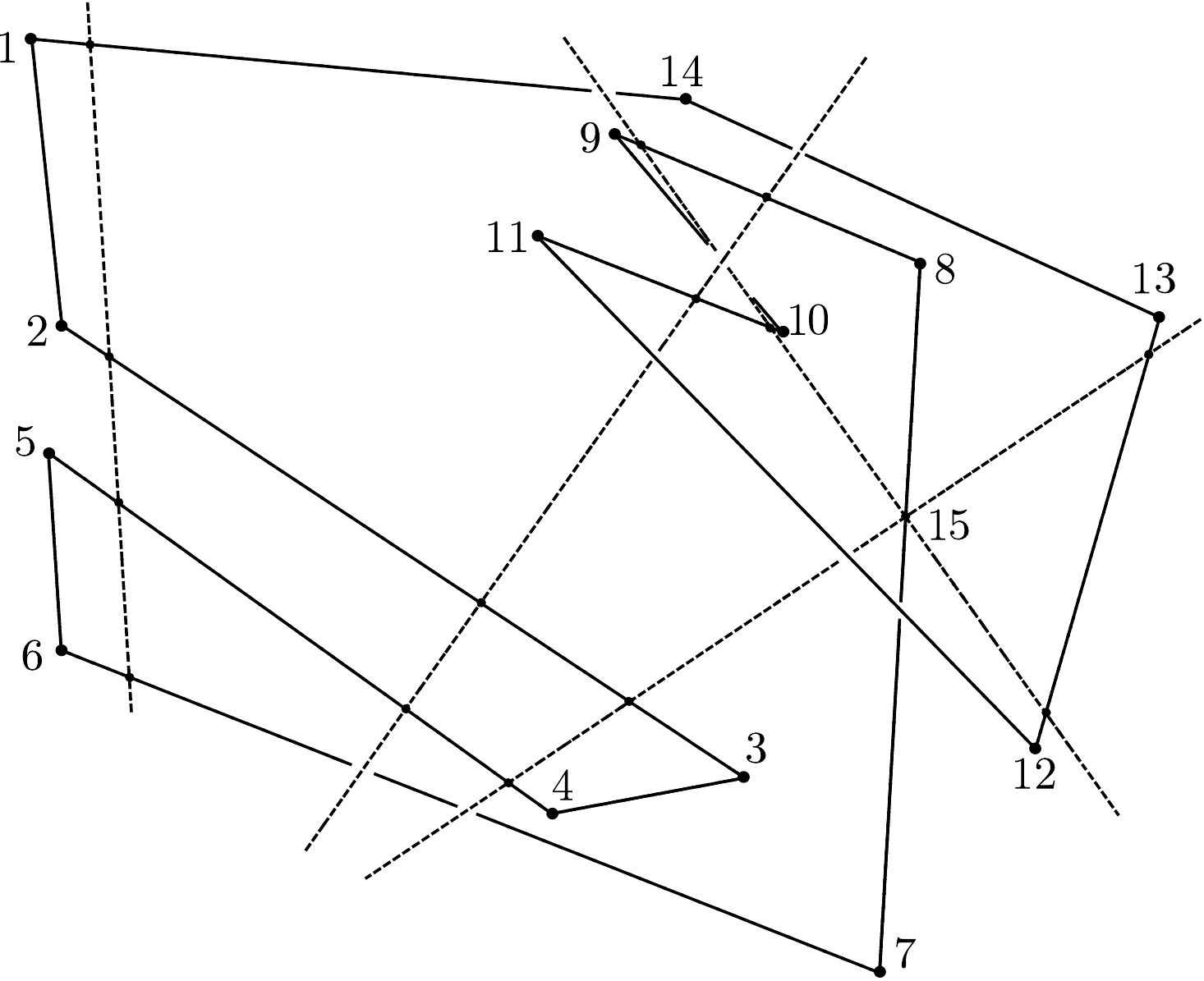}}}
\caption{The final knot $K_{14}$}\label{fig:knot1}
\end{figure}

$K_{14}$ is obtained from $K_0$ by firstly shrinking the edge $V_3V_4$ and slightly moving $V_6$ along $z$-axis, then extending edges $V_7V_6$, $V_4V_5$, $V_2V_3$, $V_{14}V_1$, $V_8V_9$, $V_{10}V_{11}$ and $V_{12}V_{13}$, according to the following formulas.
\begin{align*}
&U_3=V_3-(V_3-V_4)\delta_3, &&U_4=V_4+(V_3-V_4)\delta_4,\\
&U_6=V_6-(0,0,\delta_6), &&W_6=U_6+(U_6-V_7)\epsilon_6,\\
&W_4=U_4+(U_4-V_5)\epsilon_4, &&W_5=V_5-(U_4-V_5)\epsilon_5,\\
&W_2=V_2+(V_2-U_3)\epsilon_2, &&W_3=U_3-(V_2-U_3)\epsilon_3,\\
&W_1=V_1+(V_1-V_{14})\epsilon_1, &&W_9=V_9+(V_9-V_8)\epsilon_9,\\
&W_{10}=V_{10}+(V_{10}-V_{11})\epsilon_{10}, &&W_{11}=V_{11}-(V_{10}-V_{11})\epsilon_{11},\\
&W_{12}=V_{12}+(V_{12}-V_{13})\epsilon_{12}, &&W_{13}=V_{13}-(V_{12}-V_{13})\epsilon_{13}.
\end{align*}
Here the $\delta$'s and $\epsilon$'s are small positive numbers. The knot $K_{14}$ is obtained by the following choices:
\begin{align*}
&\delta_3=\epsilon_3=\epsilon_5=1/5, \quad\delta_4=1/4, \quad\delta_6=1/2,\\
&\epsilon_1=\epsilon_2=\epsilon_4=\epsilon_6=\epsilon_9=\epsilon_{12}=\epsilon_{13}=1/10,\\ &\epsilon_{10}=\epsilon_{11}=1/20.
\end{align*}
By computation using Mathematica, we found that $K_{14}$ satisfies the general position conditions $(a')$ and $(b')$ in Section \ref{sec:algorithm}. By the discussion in Section \ref{sec:algorithm} and Mathematica, it is not hard to see that $K_{14}$ has only $4$ quadrisecants, which are shown by the dashed lines in Figure \ref{fig:knot1}.

$K_{14}$ is a polygonal unknot. The quadrisecants and quadrisecant approximation of $K_{14}$ are given in Figure \ref{fig:knot2}. Its quadrisecant approximation $\widehat{K_{14}}$ is a trefoil knot.

\begin{figure}[h]
\centerline{\scalebox{0.6}{\includegraphics{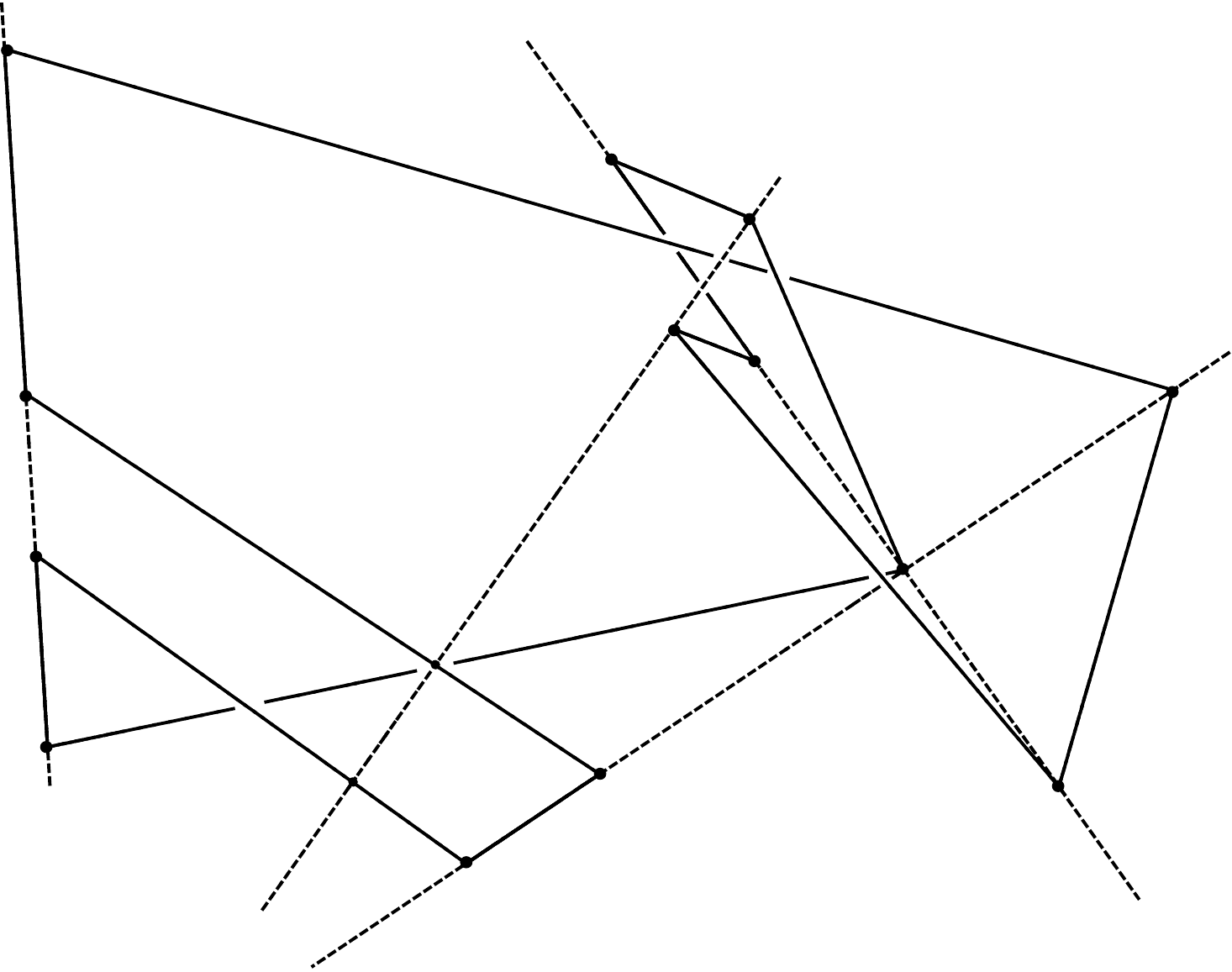}}}
\caption{The quadrisecant approximation $\widehat{K_{14}}$}\label{fig:knot2}
\end{figure}

\begin{remark}
In the above example, $K_{14}$ has two quadrisecants intersecting each other, which is not the generic case. We can perform slight perturbation to get disjoint quadrisecants while keeping all required results. For example, we can change $V_{10}$ and $V_{12}$ to $(8, 0, -1.1)$ and $(12,-2,-3.3)$ to achieve this, and the two corresponding intersection points on the edge $V_7V_8$ are $V_{15}=(10,-1,-2)$ and $V_{15}^\prime=(10,-1,-2.2)$. We can also perturb $V_3$ to get disjoint quadrisecants.
\end{remark}

\section{Quadrisecant approximation of connected sums}\label{sec:general}

In this section, we define a connected sum operation, and we use it to give counterexamples for general knot types. The counterexample will be a connected sum with one summand having a given knot type and satisfying certain conditions, and the other being the knot $K_6$ or $K_{14}$ in Section \ref{sec:unknot}.

\subsection{The connected sum operation}

Let $K=V_1V_2\cdots V_n$ be a polygonal knot in $\mathbb{R}^3$. Given a plane $\Pi$ in $\mathbb{R}^3$, let $K_\Pi$ be the image of $K$ under the perpendicular projection from $\mathbb{R}^3$ onto $\Pi$. We can regard directed projections as points on the unit sphere $S^2\subset\mathbb{R}^3$ with the induced topology.

\begin{lemma}[{\cite[Proposition 1.12]{Burde1985}}]
For a given polygonal knot $K$, the set of projections whose image has only transverse double self-intersections is open and dense in $S^2$.
\end{lemma}

Suppose that $\Pi$ is such a plane. We choose a vertex of the convex hull of $K_\Pi$, whose preimage must be some vertex, say $V$, of $K$.

\begin{lemma}\label{lem:smallK}
Let $K$, $\Pi$ and $V$ be as above, and $\Sigma\subset\mathbb{R}^3$ be a plane perpendicular to the $xy$-plane. Then for any $P\in \Sigma$ and $\epsilon, \delta>0$, there exists an affine transformation $T$ from $\mathbb{R}^3$ to itself, such that:

$(a)$ $T(\Pi)$ is parallel to the $xy$-plane.

$(b)$ $T(V)=P$ and $T(K)$ is in the $\epsilon$-neighborhood of $P$.

$(c)$ $T(K)\cap \Sigma=P$, namely $T(K)$ lies in one side of $\Sigma$.

$(d)$ If a straight line $L$ intersects more than two edges of $T(K)$, then the angle between $L$ and $T(\Pi)$ is smaller than $\delta$.
\end{lemma}

\begin{proof}
$(a)$ can be achieved by rotations. $(b)$ can be achieved by translations and linear contractions. Since the projection image of $V$ in $\Pi$ is a vertex of the convex hull of $K_\Pi$, $(c)$ can be achieved by a further rotation around the line containing $P$ and parallel to the $z$-axis. $(a)$ and $(b)$ will still hold. Then $(d)$ can be achieved by a further linear contraction along the $z$-axis.
\end{proof}

Note that condition $(d)$ means that if the angle between a straight line $L$ and the $xy$-plane is bigger than $\delta$, then for a plane perpendicular to $L$ the projection image of $K$ in it will have only transverse double self-intersections.

Suppose that $T$ is such an affine transformation. Denote $T(K)$ by $K_{\epsilon,\delta}^P$. One should keep in mind that the knot $K_{\epsilon,\delta}^P$ also depends on the plane $\Pi$, the vertex $V$, the plane $\Sigma$ and the affine transformation $T$.

\begin{lemma}\label{lem:PUnion}
Let $K$ and $K'$ be two polygonal knots in $\mathbb{R}^3$. Suppose that $P$ is a point in $K'$, such that edges containing $P$ are not parallel to the $xy$-plane. Then for sufficiently small $\epsilon, \delta>0$, we have $K_{\epsilon,\delta}^P\cap K'=P$.
\end{lemma}

\begin{proof}
Let $E$ be the union of edges of $K'$ containing $P$. It will contain one or two edges. By $(b)$ in Lemma \ref{lem:smallK}, if $\epsilon$ is sufficiently small, then $K_{\epsilon,\delta}^P\cap K'\subset E$. Since edges in $E$ are not parallel to the $xy$-plane, the angle between an edge in $E$ and the $xy$-plane is nonzero. Since edges in $E$ contain $P$, which belongs to two edges of $K_{\epsilon,\delta}^P$, by $(a)$ and $(d)$ in Lemma \ref{lem:smallK}, if $\delta$ is sufficiently small, then $E$ can not intersect $K_{\epsilon,\delta}^P-\{P\}$. Hence $K_{\epsilon,\delta}^P\cap K'=P$.
\end{proof}

Let $K$, $K'$, $P$, $E$ be as in Lemma \ref{lem:PUnion} and the proof. Denote the $\epsilon$-neighborhood of $P$ by $N_\epsilon(P)$. Then we can define the connected sum $K\#_PK'$ as following:

(i) choose $\epsilon$ sufficiently small such that $N_\epsilon(P)\cap K'\subset E$;

(ii) choose $\delta$ sufficiently small such that $K_{\epsilon,\delta}^P\cap K'=P$;

(iii) let $K_{\epsilon,\delta}^P\vee K'$ be the one point union of $K_{\epsilon,\delta}^P$ and $K'$ via $P$;

(iv) choose a way to resolve $P$ in $K_{\epsilon,\delta}^P\vee K'$, and get a one component circle.

In (iv), $P$ will be replaced by two points quite near $P$, see Figure \ref{fig:sum} (In the left picture $E$ contains one edge, and in the right picture $E$ contains two edges. In each picture the arrow shows how to resolve $P$). Clearly the knot $K\#_PK'$ has the knot type of a connected sum of $K$ and $K'$ in the usual sense.

\begin{figure}[h]
\centerline{\scalebox{0.6	}{\includegraphics{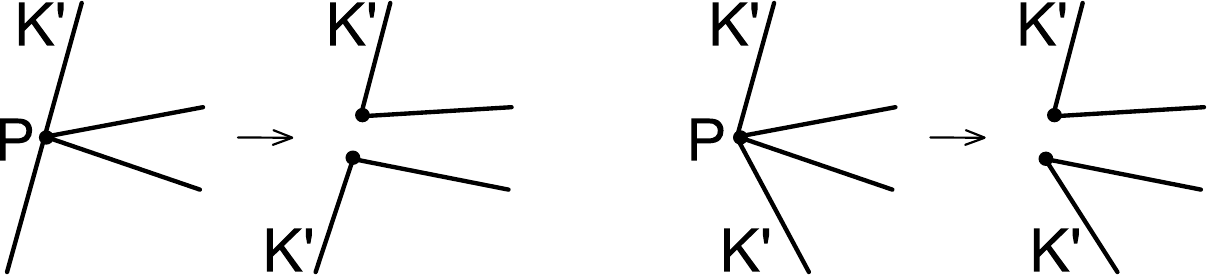}}}
\caption{From $K_{\epsilon,\delta}^P\vee K'$ to $K\#_P K'$}\label{fig:sum}
\end{figure}

\subsection{Counterexamples for general knots}

Let $K$ be a knot with $n$ edges, and let $\Pi$, $V$ be as above. When we choose $\Sigma$, $P$, $\epsilon$, $\delta$ and $T$ in Lemma \ref{lem:smallK} suitably, we can have $K_{\epsilon,\delta}^P$. Let $K_{\epsilon,\delta}^P=V_1V_2\cdots V_n$ and $V_1=P$. In following examples, $P$ will be a vertex of $K_6$ or $K_{14}$, and we will consider the knots $K\#_P K_6$ and $K\#_P K_{14}$ which have the knot type of $K$.

In each connected sum the point $P$ will be replaced by two points $P^1$ and $P^2$, and $K_{\epsilon,\delta}^P$ will become the broken line $P^1V_2\cdots V_nP^2$, which will be denoted by $\Omega$. And the interior of a broken line $M$ will be denoted by $M^\circ$.

\begin{proposition}\label{prop:general_with_self-intersection}
For any knot type $\mathcal{K}$ there is a polygonal knot $K_\ast$ of type $\mathcal{K}$ with $e(\mathcal{K})+6$ edges such that $\widehat{K_\ast}$ has self-intersections.
\end{proposition}

\begin{proof}
In $K_6$ we choose $P=W_3=(2,0,1)$, and let $\Sigma$ be the plane containing $W_2$, $W_3$, $W_4$. $P$ will be replaced by the two points $W_3^1=(2-\eta,0,1)$ and $W_3^2=(2+\eta,0,1)$, here $\eta>0$ is sufficiently small. Suppose that $P^1=W_3^1$ and $P^2=W_3^2$ such that the projection images of $P^1V_2$ and $P^2V_n$ in the $xy$-plane does not intersect. Then we can have the knot
\begin{align*}
K\#_P K_6=W_1W_2P^1V_2\cdots V_nP^2W_4W_5W_6.
\end{align*}
Let $\Gamma$ be the broken line $W_3^2W_4W_5W_6W_1W_2W_3^1$. It will have only one quadrisecant $L$. Let $L\cap W_1W_6=U_1$, $L\cap W_5W_6=U_2$, and let $\Lambda$ be the broken line $U_1W_6U_2$.

By the discussion in Section \ref{sec:algorithm}, we can perturb the vertices of $K\#_P K_6$ slightly to get a knot $K_\ast$, such that $K_\ast$ has finitely many quadrisecants. $\Omega$, $\Gamma$, $L$ and $\Lambda$ will be moved slightly. If there is a quadrisecant $L'$ of $K_\ast$ intersecting $\Lambda^\circ$, then it must intersect $\Omega^\circ$, since $L$ is the only quadrisecant of $\Gamma$. Then, if the $\delta$ in Lemma \ref{lem:smallK} is sufficiently small, then $L'\cap \Omega^\circ$ can contain at most two points.

Case 1: $L'\cap \Omega^\circ$ contains two points. By $(c)$ in Lemma \ref{lem:smallK}, if the $\epsilon$ and $\delta$ are sufficiently small, then the $y$-coordinates of all such intersection points will have a positive lower bound, and a straight line passing such points and intersecting $\Lambda^\circ$ can intersect $\Sigma$ at only one point. Then $L'$ can not be a quadrisecant of $K_\ast$.

Case 2: $L'\cap \Omega^\circ$ contains one point. Then $L'\cap \Sigma$ contains three points, and $L'$ must intersect $W_2P^1$ and $W_4P^2$. And it can not intersect $\Omega^\circ$, by $(c)$ in Lemma \ref{lem:smallK}.

The contradictions mean that no quadrisecant of $K_\ast$ can intersect $\Lambda^\circ$, and $U_1U_2$ will be an edge of $\widehat{K_\ast}$. Then $\widehat{K_\ast}$ will have self-intersections.
\end{proof}

\begin{proposition}\label{prop:general_no_trefoil}
For any knot type $\mathcal{K}$, there is a polygonal knot $K_\diamond$ of type $\mathcal{K}$ with $e(\mathcal{K})+14$ edges, such that $\widehat{K_\diamond}$ is a connected sum with the (left-handed) trefoil knot as a summand.
\end{proposition}

\begin{proof}
In $K_{14}$ we choose $P=W_7=(10,-1,-8)$, and let $\Sigma$ be the plane containing $W_6$, $W_7$, $W_8$. $P$ will be replaced by the two points $W_7^1=(10-10\eta,\eta-1,-8)$ and $W_7^2=W_7$, here $\eta>0$ is sufficiently small. Suppose $P^1=W_7^1$ and $P^2=W_7^2$ such that the projection images of $P^1V_2$ and $P^2V_n$ in the $xy$-plane does not intersect. Then we can have the knot
\begin{align*}
K\#_P K_{14}=W_1W_2\cdots W_6P^1V_2\cdots V_nP^2W_8W_9\cdots W_{14}.
\end{align*}
Let $\Gamma$ be the polygonal knot $W_7^2W_8W_9\cdots W_{14}W_1W_2\cdots W_6W_7^1$. It will have the same set of quadrisecants of $K_{14}$, namely $L_1$, $L_2$, $L_3$ and some $L_4$ described in Section \ref{sec:unknot}. Let $L_1\cap W_1W_{14}=U_1$, $L_2\cap W_{12}W_{13}=U_2$, and let $\Lambda$ be the broken line $U_1W_{14}W_{13}U_2$.

By the discussion in Section \ref{sec:algorithm}, we can perturb the vertices of $K\#_P K_{14}$ slightly to get a knot $K_\diamond$, such that $K_\diamond$ has finitely many quadrisecants. $\Omega$, $\Gamma$, $\Lambda$ and the four quadrisecants $L_1$, $L_2$, $L_3$, $L_4$ will be moved slightly. If there is a quadrisecant $L'$ of $K_\diamond$ intersecting $\Lambda^\circ$, then it must intersect $\Omega^\circ$, since no quadrisecant of $\Gamma$ can intersect $\Lambda^\circ$. Then, if the $\delta$ in Lemma \ref{lem:smallK} is sufficiently small, then $L'\cap \Omega^\circ$ contain at most two points.

Case 1: $L'\cap \Omega^\circ$ contains two points. By $(c)$ in Lemma \ref{lem:smallK}, if the $\epsilon$ and $\delta$ are sufficiently small, then a straight line passing such points and intersecting $\Lambda^\circ$ can intersect $\Sigma$ at only one point. And $L'$ can not be a quadrisecant of $K_\diamond$.

Case 2: $L'\cap \Omega^\circ$ contains one point. Then $L'\cap \Sigma$ contains three points, and $L'$ must intersect $W_6P^1$ and $W_8P^2$. And it can not intersect $\Omega^\circ$, by $(c)$ in Lemma \ref{lem:smallK}.

The contradictions mean that no quadrisecant of $K_\diamond$ can intersect $\Lambda^\circ$, and $U_1U_2$ will be an edge of $\widehat{K_\diamond}$. Since $\Omega\subset N_\epsilon(P)$, $\widehat{K_\diamond}$ will be a connected sum of a trefoil knot and some knot $K'$.
\end{proof}

If the knot $K$ in the above proposition does not contain the trefoil knot as a connected summand (for example if $K$ is a prime knot other than the trefoil knot), then $\widehat{K_\diamond}$ will have the knot type different from $K$. For arbitrary knot type we have the following.

\begin{proposition}\label{prop:general}
For any polygonal knot $K$ with $n$ edges, there is a polygonal knot $K_\diamond^1$ with the same knot type as $K$ and $5\floor*{\frac{n+1}2}+14$ edges, such that $\widehat{K_\diamond^1}$ has the knot type of a connected sum of $K$ and a (left-handed) trefoil knot.
\end{proposition}

\begin{proof}
The example of $K_6$ in Section \ref{sec:unknot} means that if we suitably replace a line segment in a chosen edge of $K$ by two edges and suitably perturb the vertices, then there will be a quadrisecant $L$ such that the union of line segments in $L$ between the secant points is quite close to the chosen edge, as illustrated in Figure \ref{fig:general_knot_1}.
\begin{figure}[htbp]
\begin{center}
\subfloat[]{
\label{fig:general_knot_1}
\includegraphics{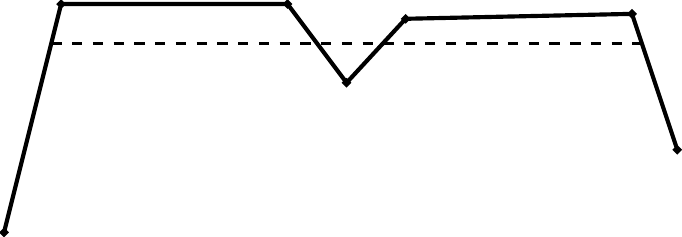}}
\hspace{20pt}
\subfloat[]{
\label{fig:general_knot_2}
\includegraphics{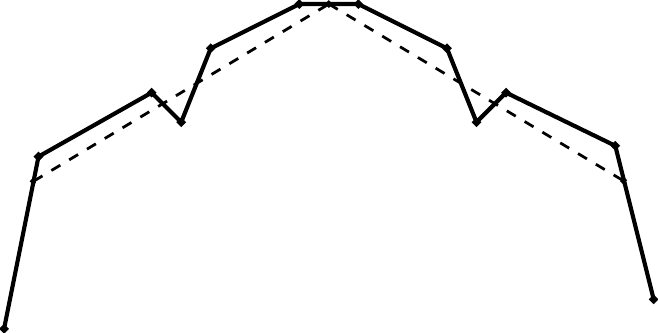}}
\end{center}
\caption{General knots. The dashed edge is the original edge, while the solid edges are the new edges.}\label{fig:general_knot}
\end{figure}
Here `quite close' means that the Hausdorff distance between the two sets is sufficiently small.

We do this procedure successively for every other edge of $K$, such that each time the replacement happens in a sufficiently small neighborhood of the edge. When $n$ is odd, we need change two adjacent edges and there will be one more edge at their comment vertex, as in Figure \ref{fig:general_knot_2}. We get a knot $K^1$ sufficiently close to $K$ such that $\widehat{K^1}$ is also quite close to $K$. In particular, both $K^1$ and $\widehat{K^1}$ will have the same knot type of $K$.

$K^1$ has $5\floor*{\frac{n+1}2}$ edges. By the construction as in Proposition \ref{prop:general_no_trefoil}, we can get a $K_\diamond^1$ from $K^1$. Since quadrisecants will be preserved under affine transformations, the $\Omega$ part will not change too much under the quadrisecant approximation. Then $\widehat{K_\diamond^1}$ has the knot type of a connected sum of $K$ and the left-handed trefoil knot.
\end{proof}

\begin{proof}[Proof of Theorem \ref{thm:main}] It evidently follows from Propositions \ref{prop:general_with_self-intersection}, \ref{prop:general_no_trefoil} and \ref{prop:general}.
\end{proof}

\bibliographystyle{amsalpha}

\end{document}